\documentclass[reqno,11pt]{amsart}
\usepackage{color}
\usepackage{mathtools}
\usepackage{amsfonts}
\usepackage{amssymb}
\usepackage{amsthm}
\usepackage{newlfont}
\usepackage{graphicx}
\usepackage{float}
\usepackage{amsmath}
\usepackage{multirow}

\usepackage{geometry}
\usepackage{tikz}
\usetikzlibrary{matrix,shapes,arrows,positioning,chains}

\numberwithin{equation}{section}
\newtheorem{thm}{Theorem}[section]
\newtheorem{lem}{Lemma}[section]
\newtheorem{rem}{Remark}[section]

\newcommand{\eps}{\varepsilon}

\newcommand{\artanh}{\operatorname {artanh}}

\newcommand{\ed}{\end {document}}

\begin{document}
\sloppy

\title[Stability and Convergence of Strang splitting method]{Stability and convergence of Strang splitting. Part I: Scalar Allen-Cahn equation}

\author[D. Li]{ Dong Li}
\address{D. Li, SUSTech International Center for Mathematics, and Department of Mathematics,  Southern University of Science and Technology,
	Shenzhen, China }
\email{lid@sustech.edu.cn}

\author[C.Y. Quan]{Chaoyu Quan}	
\address{C.Y. Quan, SUSTech International Center for Mathematics,  Southern University of Science and Technology,
	Shenzhen, China }
\email{quancy@sustech.edu.cn}

\author[J. Xu]{Jiao Xu}	
\address{J. Xu, SUSTech International Center for Mathematics,  Southern University of Science and Technology,
	Shenzhen, China }
\email{xuj7@sustech.edu.cn}


\maketitle

\begin{abstract}
We consider a class of second-order Strang splitting methods for Allen-Cahn equations with polynomial
or logarithmic nonlinearities. For the polynomial case both the linear and the nonlinear propagators
are computed explicitly.  We show that this type of Strang splitting scheme is
unconditionally stable regardless of the time step. Moreover we establish strict energy dissipation
for a judiciously modified energy which coincides with the classical energy up to $\mathcal O(\tau)$ where
$\tau$ is the time step. 
For the logarithmic potential case, since the continuous-time nonlinear propagator no longer enjoys
explicit analytic treatments, we employ a second order in time two-stage implicit Runge--Kutta (RK) nonlinear propagator together with an efficient Newton iterative solver.   We prove a maximum principle which ensures phase separation and
establish energy dissipation law under mild restrictions on the time step. These appear to be the first
rigorous results on the energy dissipation of Strang-type splitting methods for Allen-Cahn equations.
\end{abstract}

\section{Introduction}
In this work we consider the Allen-Cahn equation \cite{AC79} of the form
\begin{align} \label{1.1}
\begin{cases}
\partial_t u  = \varepsilon^2 \Delta  u-  f(u), \qquad (t,x) \in (0,\infty) \times \Omega; \\
u\Bigr|_{t=0} =u^0,
\end{cases}
\end{align}
where $u$ is a real-valued function corresponding to the concentration of a phase in a multi-component alloy, and $u^0$ is the initial condition.  For simplicity we take the spatial domain $\Omega$ to be the $2\pi$-periodic
torus $\mathbb T^d=[-\pi, \pi]^d$ in physical dimensions $d\le 3$. With some minor work
our analysis can be extended to many other situations.
The parameter $\varepsilon^2>0$ is the mobility coefficient which is fixed as
a constant. In its present non-dimensionalized form  the magnitude of $\varepsilon$  
governs the typical length scale of an interface in the dynamical evolution. 
 The nonlinear term $f(u)$ is taken as the derivative of a given potential function,
namely $f(u)= F^{\prime}(u)$.  We will be primarily concerned with two typical potential
functions. One is the standard double-well potential
\begin{align}
 F_{\mathrm{st}}(u)  = \frac 14 (u^2-1)^2
 \end{align}
whose extrema $u = \pm 1$ correspond to two different phases. The other is the logarithmic Flory--Huggins free energy  \cite{F42, H41}
\begin{equation}\label{eq:logpot}
F_{\mathrm{fh}} (u) = \frac{\theta}{2}[(1+u)\ln(1+u) + (1-u)\ln(1-u)] - \frac{\theta_c}{2} u^2,
\end{equation}
where $0<\theta<\theta_c$ denote the absolute temperature and the critical temperature respectively.
The condition $0<\theta<\theta_c$ is very physical since it ensures that $F_{\mathrm{fh}}$
has a double-well form with two equal minima situated at $u_+$ and $-u_+$, where $u_+>0$
is the positive root of the equation 
\begin{align}
0=F_{\mathrm{fh}}^{\prime}(u) =\frac{\theta}2 \ln \frac{1+u}{1-u} -\theta_c u.
\end{align}
 When the quenching
is shallow (i.e. $\theta$ is close to $\theta_c$), one can Taylor-expand near $u=0$
and obtain the standard polynomial approximation of the free energy. For smooth solutions of \eqref{1.1}, we have the energy dissipation 
\begin{align}
\frac {d} {dt} E(u) \le 0,
\end{align}
where
\begin{align}
E(u) = \int_{\Omega} \Bigl( \frac {\varepsilon^2} 2 |\nabla u |^2 + F(u)  \Bigr) dx,
\end{align}
and $F(u) =F_{\mathrm{st}}(u)$ or $F_{\mathrm{fh}}(u)$.  In practical numerical simulations,
the energy dissipation law is often used as a fidelity check of the algorithm. 

In this work we shall analyze the stability of second-order in time splitting methods applied
to the Allen-Cahn equation. Due to its simplicity  the operator splitting methods have 
found its ubiquitous presence in the numerical simulation of many physical problems,
including phase-field equations \cite{CKQT15, WT16, LQ1a, LQ1b, LQ1c, L21}, 
Schr\"odinger equations \cite{BJM02, T12, LW20}, and  the reaction-diffusion systems \cite{D01, Nie11}.  A prototypical second order in time method is the Strang splitting approximation
\cite{St68, M90}. Specifically for the Allen--Cahn equation under study,  we adopt the following
Strang splitting discretization 
\begin{equation}\label{eq:Strang}
u^{n+1} = \mathcal S_{\mathcal L}\left( \tau/2\right) \mathcal S_{\mathcal N}\left( \tau \right)\mathcal S_{\mathcal L}\left( \tau/2\right) u^n,
\end{equation}
where $\tau>0 $ denotes the time step,  and $S_{\mathcal L}(\frac 12\tau) 
=\exp( \varepsilon^2 \frac 12\tau \Delta)$ is  the linear propagator.  The nonlinear propagator 
$S_{\mathcal N}(\tau): a\mapsto u(\tau)$ is the nonlinear solution operator of the system
\begin{align} \label{1.8}
\begin{cases}
\partial_t  u = -f(u),\\
u\Bigr|_{t=0}=a.
\end{cases}
\end{align}
Denote  by $S_{\mathrm{ex}}(\tau)$ the exact nonlinear solution operator to \eqref{1.1}.  The 
propagator \eqref{eq:Strang} is a second order in time approximation in the sense that it admits
\begin{align}
&\text{$\mathcal O(\tau^3)$ one-step approximation error}: \qquad S_{\mathrm{ex}}(\tau) u^n = u^{n+1} + \mathcal O(\tau^3);  \label{1.9a} \\
& \text{$\mathcal O(\tau^2)$ in-$\mathcal O(1)$-time approximation error}: \qquad 
\sup_{n\tau \le T} \| u^n - S_{\mathrm{ex}} (n\tau) u^0 \| = \mathcal O(\tau^2). \label{1.9b}
\end{align}
Here $[0, T]$ is a given compact time interval,  $\|\cdot \|$ is some Sobolev norm and the implied
constants in $\mathcal O(\tau^2)$ can depend on $T$.  
As it turns out the numerical performance
of the scheme \eqref{eq:Strang} is quite good for solving the Allen-Cahn equation
\cite{WT16}. On the other hand, it should be noted that the somewhat heuristic estimates
\eqref{1.9a}--\eqref{1.9b}  rest on various subtle regularity assumptions on the exact 
solution and the numerical iterates.  A fundamental open issue is to establish the stability and regularity 
of the Strang splitting solutions in various Sobolev classes.  The very purpose of this paper is
to settle this problem for the Allen-Cahn equation \eqref{1.1} with the polynomial or the logarithmic potential nonlinearities. Our first result is concerned with the polynomial case.
Note that in this case the nonlinear propagator $S_{\mathcal N}(\tau)$ can be expressed
explicitly. 

\begin{thm}[Stability of Strang-splitting for AC, polynomial case] \label{thm1}
Let $\varepsilon>0$, $d\le 3$ and consider \eqref{1.1} on the periodic
torus $\mathbb T^d=[-\pi, \pi]^d$ with $f(u) = u^3-u$.  Let $\tau>0$ and
denote $S_{\mathcal L}(\tau) = \exp(\varepsilon^2 \tau \Delta)$.  Denote $S_{\mathcal N}(\tau)$
according to \eqref{1.8}.  Consider the Strang splitting discretization 
\begin{equation}
u^{n+1} = \mathcal S_{\mathcal L}\left( \tau/2\right) \mathcal S_{\mathcal N}\left( \tau \right)\mathcal S_{\mathcal L}\left( \tau/2\right) u^n, \qquad n\ge 0.
\end{equation}
The following hold.

\begin{enumerate}
\item \underline{The maximum principle}. For any $\tau>0$ and any $n\ge 0$, it holds that
\begin{align}
\| u^{n+1} \|_{\infty} \le \max \{ 1, \; \|u^n \|_{\infty} \}.
\end{align}
It follows that
\begin{align}
\sup_{n\ge 1} \| u^n \|_{\infty} \le \max\{1, \; \| u^0 \|_{\infty} \}.
\end{align}
In  particular if $\| u^0 \|_{\infty} \le 1$, then
\begin{align}
\sup_{n\ge 1} \|u^n \|_{\infty} \le 1.
\end{align}

\item \underline{Modified energy dissipation}. Let $u^0 \in H^1(\mathbb T^d)$. For any $\tau>0$ and any $n\ge 0$, we have
\begin{align}
\widetilde E^{n+1} \le \widetilde E^n.
\end{align}
Here (below $\langle, \rangle$ denotes the usual $L^2$ inner product)
\begin{align}
&\widetilde E^n= \frac 1{2\tau} \langle 
(1-e^{\varepsilon^2 \tau \Delta})u^n, \, u^n \rangle
+ \int_{\mathbb T^d} \widetilde F(\tilde u^n )dx  \\
& \quad\;  = \frac 1 {2\tau}
\langle (e^{-\varepsilon^2 \tau \Delta} -1) \tilde u^n, \, \tilde u^n \rangle
+ \int_{\mathbb T^d} \widetilde F(\tilde u^n) dx; \\
&\;\;\tilde u^n  = \mathcal S_{\mathcal L}(\tau/2) u^{n}; \\
& \;\;\widetilde F(\tilde u^n)  = \frac 14+\frac1{2\tau} (\tilde u^n)^2 - \frac{e^\tau}{\tau(e^{2\tau}-1)} \left(\sqrt{1+(e^{2\tau}-1)(\tilde u^n)^2}-1\right).\label{1.19}
\end{align}

\item \underline{Uniform Sobolev bounds}.  Let $u^0 \in H^{k_0}(\mathbb T^d)$ for some
$k_0\ge 1$. It holds that
\begin{align}
\sup_{n\ge 1} \| u^n \|_{H^{k_0}(\mathbb T^d)} \le C_1,
\end{align}
where $C_1>0$ depends only on ($\eps$, $k_0$, $d$, $\|u^0 \|_{H^{k_0}}$). 
Moreover for any $k\ge k_0$, we have
\begin{align}
\sup_{n\ge \frac {1} {\tau} } \| u^n \|_{H^{k} (\mathbb T^d) }
\le C_2,
\end{align}
where $C_2>0$ depends only on ($\eps$, $k$, $k_0$, $d$, $\| u^0 \|_{H^{k_0} } $ ).

\item \underline{Connection with the standard energy}.  Let $u^0$ be smooth (for example $u^0 \in H^{20}(\mathbb T^d)$).
For $0<\tau\le 1$, we have
\begin{align} \label{Eu_1.22}
\sup_{n\ge 0} | \widetilde E^n - E(u^n ) | \le C_3 \tau,
\end{align}
where $C_3>0$ depends only on ($\eps$, $d$, $u^0$).

\item \underline{Uniform second order approximation}. Assume the initial data $u^0$ is sufficiently smooth (for example $u^0 \in H^{40}(\mathbb T^d)$).  Let $u$ be the exact PDE solution
to \eqref{1.1}  corresponding to initial data $u^0$. Let $0<\tau \le 1$.
 Then for any $T>0$, we have
\begin{align} \label{Eu_1.23}
\sup_{n\ge 1, n\tau \le T}  \| u^n - u(n\tau, \cdot ) \|_{L^2(\mathbb T^d)}
\le C \cdot \tau^2,
\end{align}
where $C>0$ depends on ($\eps$, $u^0$, $T$).

\end{enumerate}
\end{thm}
\begin{rem}
Consider the function 
\begin{align}
h(x) = \frac 14+\frac1{2\tau} x - \frac{e^\tau}{\tau(e^{2\tau}-1)} \left(\sqrt{1+(e^{2\tau}-1)x}-1\right),\; x\in[0,\infty).
\end{align}
Clearly, $h(0) = \frac14$, $h(x)\rightarrow \infty$ as $x\rightarrow \infty$, and 
\begin{align}
h^\prime (x) = 0 \Leftrightarrow x=1.
\end{align}
We have 
\begin{align}
h(1) & = \frac14+\frac1{2\tau}-\frac{e^\tau}{\tau(e^\tau+1)}\\
& = \frac14-\frac1{2\tau}\frac{e^\tau-1}{e^\tau+1} \ge 0,\quad \forall \, 0\le\tau<\infty.
\end{align}
Thus $\widetilde F$ defined in \eqref{1.19} is always nonnegative. 
On the other hand, by using Fourier transform, we have
\begin{align}
\langle (e^{-\varepsilon^2 \tau \Delta} -1) w, \, w\rangle = c_d \sum_{0\ne k\in \mathbb Z^d} (e^{\varepsilon^2 \tau |k|^2} -1) |\widehat w(k)|^2\ge 0,
\end{align}
where $c_d>0$ depends only on the dimension $d$. 
Therefore $\widetilde E$ always stays nonnegative. 

\begin{figure}[!]
\includegraphics[trim = {0in 0.in 0.5in 0.2in},clip,width = 0.5\textwidth]{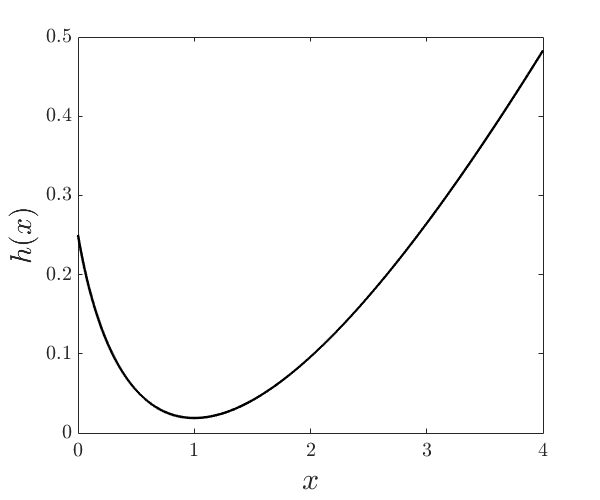}
\caption{$h(x)$ w.r.t. $x$ for $\tau = 1$.}\label{fig_h}
\end{figure}
\end{rem}

\begin{rem}
The regularity assumptions in \eqref{Eu_1.22} and \eqref{Eu_1.23} can be lowered.
However for simplicity of presentation we do not dwell on this issue in this work.
\end{rem}

Our second result focuses on the AC equation with the logarithmic potential \eqref{eq:logpot}, i.e.
\begin{align} \label{Eu_1.24}
\begin{cases}
\partial_t u = \eps^2 \Delta u - f^{\mathrm{LOG}}(u), 
\qquad f^{\mathrm{LOG}}(u) = -\theta_c u +\frac{\theta}2 \ln \frac{1+u}{1-u}; \\
u\Bigr|_{t=0} =u^0,
\end{cases}
\end{align}
where $0<\theta<\theta_c$.  It is not difficult to check that (cf. the analysis after \eqref{Eu_3.7})
$f^{\mathrm{LOG}}(u)$ admits a unique root in the interval $(0,1$) which we denote as
$u_*$.  For smooth solutions to \eqref{Eu_1.24}, we have the maximum principle:
$\| u(t,\cdot )\|_{\infty} \le u_*$  for all $t>0$ if $\| u_0 \|_{\infty} \le u_*$. On the other hand, it is a nontrivial task
to design suitable numerical discretization preserving this important maximum principle.

By direct analogy with the polynomial potential case, one can consider the exact Strang-type second order in time splitting scheme:
\begin{align}
u^{n+1} = \mathcal S_{\mathcal L}(\tau /2) \mathcal S_{\mathcal N}^{(\mathrm{LOG})}(\tau)
\mathcal S_{\mathcal L}(\tau/2) u^n,
\end{align}
where  $\mathcal S_{\mathcal L}(\tau) =
\exp(\eps^2 \tau \Delta)$ and 
  $\mathcal S_{\mathcal N}^{(\mathrm{LOG})} (\tau)$ is the solution operator $w_0\to w(\tau)$ of the equation
\begin{align}
\begin{cases}
\partial_t w = \theta_c w-\frac {\theta}2 \ln \frac{1+w}{1-w}, \quad 0<t\le \tau; \\
w\Bigr|_{t=0}=w_0.
\end{cases}
\end{align}
However  a pronounced difficulty with the implementation of the above scheme is the lack
of an explicit solution formula for the solver $\mathcal S_{\mathcal N}^{(\mathrm{LOG})} (\tau)$. 
To solve this problem we  approximate $\mathcal S_{\mathcal N}^{(\mathrm{LOG})} (\tau) $ by a further judiciously  chosen numerical discretization.  The choice of the numerical solver turns out to be  rather subtle and technically involved, since one has to control the truncation error to be within $\mathcal O(\tau^3)$ and preserve the maximum principle at the same time {(see the recent deep work of Li, Yang, and Zhou \cite{LYZ} where an ingenious cut-off procedure is developed)}.

To approximate $\mathcal S_{\mathcal N}^{(\mathrm{LOG})}(\tau)v$ for given $v$, we adopt the Pareschi and Russo's two-stage diagonally implicit Runge Kutta (PR-RK) method \cite{PR05}:
\begin{equation}
 \begin{aligned}
& u_1 = v+ a\tau f(u_1),\\
& u_2 = v+(1-2a)\tau f(u_1)+a \tau f(u_2),\\
& \mathcal S_{\mathcal N}^{(\mathrm{LOG})}(\tau) v
\approx \widetilde{\mathcal S}_{\mathcal N}(\tau)  v : = v+ \frac{1}{2}\tau f(u_1) + \frac{1}{2}\tau f(u_2).
\end{aligned}
\end{equation}
In the above $a$ is a tunable real-valued parameter.    

 We employ
the following RK-based Strang-type splitting for \eqref{Eu_1.24}:
\begin{align} \label{Eu_1.28}
&u^{n+1} = \mathcal S_{\mathcal L}(\tau/2)  \widetilde{\mathcal S}_{\mathcal N}(\tau) \mathcal S_{\mathcal L}(\tau/2) u^n.
\end{align}
In terms of  $ \tilde u^n = \mathcal S_{\mathcal L}(\tau/2) u^n$, we have:
\begin{align} \label{Eu_1.29}
\tilde u^{n+1} = \mathcal S_{\mathcal L}(\tau) 
 \widetilde{\mathcal S}_{\mathcal N}(\tau) \tilde u^n.
\end{align}

We have the following stability result concerning the logarithmic case.
\begin{thm}[Stability of RK-based Strang-splitting for AC, logarithmic case] \label{thm2}
Let $\varepsilon>0$, $d\le 3$, $0<\theta<\theta_c$ and consider \eqref{Eu_1.24} on the periodic
torus $\mathbb T^d=[-\pi, \pi]^d$. Recall $u_*$ is the unique root of 
$f^{\mathrm{LOG}}(u)$  in the interval $(0,1$).  Let $\tau>0$ and consider the RK-based
Strang-splitting scheme defined in \eqref{Eu_1.28} and equivalently expressed in terms of $\tilde u^n$ in \eqref{Eu_1.29}. Assume $u^0 \in H^1(\mathbb T^d)$ and  $\| u^0 \|_{\infty} \le u_*$.
Assume $ a\ge 1+\frac{\sqrt 2}{2}$ and $0<\tau\le \frac{1}{3a(\theta_c-\theta)}$.
The following hold.

\begin{enumerate}
\item \underline{The maximum principle}.  It holds that
\begin{align}
 \sup_{n\ge 1} \max\{ \| u^n \|_{\infty}, \; \| \tilde u^n \|_{\infty} \} \le u_*<1.
 \end{align}
\item  \underline{Modified energy dissipation}. We have 
\begin{align} 
& \sup_{n\ge 1} \max\{ \| u^n \|_{\infty}, \; \| \tilde u^n \|_{\infty} \} \le u_*; \\
 &\overline E^{n+1} \le \overline E^n,\quad \forall\; n\ge 1;  \\
&\overline E^n \coloneqq \frac1{2}\left\langle \frac{1}{\tau}(e^{-\eps^2\tau \Delta} -1) \tilde u^n, \tilde u^n \right\rangle + \int_{\mathbb T^d} \overline F(\tilde u^n) dx,
\end{align}
where   $\overline F$ is defined by \eqref{eq:F_log}.

\item \underline{Uniform Sobolev bounds}.  Let $u^0 \in H^{k_0}(\mathbb T^d)$ for some
$k_0\ge 1$. It holds that
\begin{align}
\sup_{n\ge 1} \| u^n \|_{H^{k_0}(\mathbb T^d)} \le B_1,
\end{align}
where $B_1>0$ depends only on ($\eps$, $k_0$, $d$, $\|u^0 \|_{H^{k_0}}$, $\theta$,
$\theta_c$). 
Moreover for any $k\ge k_0$, we have
\begin{align}
\sup_{n\ge \frac {1} {\tau} } \| u^n \|_{H^{k} (\mathbb T^d) }
\le B_2,
\end{align}
where $B_2>0$ depends only on ($\eps$, $k$, $k_0$, $d$, $\| u^0 \|_{H^{k_0} } $,
$\theta$,
$\theta_c$ ).

\item \underline{Connection with the standard energy}.  Let $u^0$ be smooth (for example $u^0 \in H^{20}(\mathbb T^d)$).
For $0<\tau\le 1$, we have
\begin{align} \label{Eu_1.36}
\sup_{n\ge 0} | \overline E^n - E(u^n ) | \le B_3 \tau,
\end{align}
where $B_3>0$ depends only on ($\eps$, $d$, $u^0$, $\theta$,
$\theta_c$).

\item \underline{Uniform second order approximation}. Assume the initial data $u^0$ is sufficiently smooth (for example $u^0\in H^{40}(\mathbb T^d)$).  Let $u$ be the exact PDE solution
to \eqref{Eu_1.24}  corresponding to initial data $u^0$. Let $0<\tau \le 1$.
 For any $T>0$, we have
\begin{align} \label{Eu_1.37}
\sup_{n\ge 1, n\tau \le T}  \| u^n - u(n\tau, \cdot ) \|_{L^2(\mathbb T^d)}
\le \widetilde C \cdot \tau^2,
\end{align}
where $\widetilde C>0$ depends on ($\eps$, $u^0$, $T$, $\theta$,
$\theta_c$).
\end{enumerate}
\end{thm}
\begin{rem}
More generally, one can also show that if $\| u^0 \|_{\infty} \le \beta <1$ for some
$\beta \in [u_*,1)$, then $\| u^n \|_{\infty} \le \beta$ for all $n\ge 1$. This slightly more maximum principle covers most common cases in practical simulations.  In many cases $u_*$ is already 
close to the limit value $1$. For example  if $\theta=\frac 14$, $\theta_c=1$ and $\eps=0.01$, then
$u_*\approx 0.99933$ which already serves as a good upper bound from a practical point of view.
\end{rem}

The statement (3)--(5) in Theorem \ref{thm2} can be proved in a similar way as in the polynomial case and we omit
the repetitive details. The rest of this paper is organized as follows.
In Section \ref{sect2}, we give the proof of Theorem \ref{thm1}. 
Section \ref{sect3} is devoted to the proof of Theorem \ref{thm2}. As mentioned above we focus
on proving statement (1)--(2) in Theorem \ref{thm2}. We give detailed exposition and motivation for
these results therein. In Section \ref{sect4}, we carry out extensive numerical simulations to  showcase the stability and convergence of the Strang-splitting methods for both the polynomial
and the logarithmic cases.  The last section contains some  concluding remarks.

\section{Proof of Theorem \ref{thm1} }\label{sect2}
In this section we carry out the proof of Theorem \ref{thm1}. We divide the proof into several
steps.
\subsection{The maximum principle}
Recall that $u^{n+1} = \mathcal S_{\mathcal L}\left( \tau/2\right) \mathcal S_{\mathcal N}\left( \tau \right)\mathcal S_{\mathcal L}\left( \tau/2\right) u^n$.  
Since $ \|S_{\mathcal L}(\tau) a \|_{\infty} \le \| a\|_{\infty}$ for any $\tau>0$, we only need to examine $S_{\mathcal N}(\tau)$.
By definition, the nonlinear solver is
\begin{align} \label{eu_2.1}
\begin{cases}
\partial_t u = u -u^3, \quad 0<t \le \tau; \\
u\Bigr|_{t=0}=a.
\end{cases}
\end{align}
Thanks to the explicit polynomial nonlinearity, one can solve the above equation explicitly and obtain
\begin{equation}
u(t) = \frac{e^{t} a}{\sqrt{1+ (e^{2t} - 1)a^2}}.
\end{equation}
This renders the solution operator $S_{\mathcal N}(\tau)$ as
\begin{equation}\label{eq:SN_poly}
\mathcal S_{\mathcal N}(\tau) a = \frac{e^\tau a}{\sqrt{1+ (e^{2\tau} - 1)a^2}}.
\end{equation}
By \eqref{eq:SN_poly}, we have 
\begin{align} \label{eu_2.4}
\| \mathcal S_{\mathcal N}(\tau) a \|_{\infty} \le \max\{ \|a\|_{\infty}, 1 \}.
\end{align}
This yields the desired maximum principle. Note that one can  also work
directly with \eqref{eu_2.1} to derive  \eqref{eu_2.4}.

\subsection{Modified energy dissipation}
Since $u^{n+1} = \mathcal S_{\mathcal L}\left( \tau/2\right) \mathcal S_{\mathcal N}\left( \tau \right)\mathcal S_{\mathcal L}\left( \tau/2\right) u^n$ and
$\tilde u^n = \mathcal S_{\mathcal L} \left(\tau/2\right) u^n$, we have
\begin{equation}
\tilde u^{n+1} = \mathcal S_{\mathcal L}\left( \tau\right) \mathcal S_{\mathcal N}\left( \tau \right) \tilde u^n.
\end{equation}
This yields
\begin{equation}
e^{-\eps^2\tau\Delta}\tilde u^{n+1} =  \frac{e^\tau \tilde u^n}{\sqrt{1+(e^{2\tau}-1) (\tilde u^n)^2}}.
\end{equation}
We rewrite the above as
\begin{equation} \label{eu_2.7}
\frac 1 {\tau} \left(e^{-\eps^2\tau\Delta}-1\right)\tilde u^{n+1}  +\frac 1{\tau} (\tilde u^{n+1}-\tilde u^n) =
\frac 1 {
\tau} \left(   \frac{e^\tau \tilde u^n}{\sqrt{1+(e^{2\tau}-1) (\tilde u^n)^2}} -\tilde u^n 
\right)=-\widetilde F'(\tilde u^n),
\end{equation}
where 
\begin{equation} \label{eu_2.8}
 \widetilde F(z)  = \frac 14+\frac1{2\tau} z^2 - \frac{e^\tau}{\tau(e^{2\tau}-1)} \left(\sqrt{1+(e^{2\tau}-1)z^2}-1\right).
\end{equation}
The harmless constant $1/4$ is inserted here so that $\widetilde F$ coincides with the standard
energy when $\tau\to 0$.
Observe that 
\begin{align} \label{eu_2.9}
\tilde F( \tilde u^{n+1}) = \tilde F
(\tilde u^n) + \tilde F^{\prime}(\tilde u^n) (\tilde u^{n+1} -\tilde u^n)
+ \frac 1 2 \tilde F^{\prime\prime}(\xi^n) (\tilde u^{n+1} -\tilde u^n)^2,
\end{align}
where $\xi^n$ is some function between $\tilde u^n$ and $\tilde u^{n+1}$. 
Also
\begin{align}
  & \frac 1 {\tau} \langle (e^{-\varepsilon^2\tau \Delta} -1) \tilde u^{n+1}, \tilde u^{n+1} -\tilde u^n \rangle \notag \\
=& \frac 1 {\tau} \langle ( 1-e^{\varepsilon^2\tau \Delta}) u^{n+1}, \, u^{n+1} -u^n \rangle
\notag \\
=& \frac 1{2\tau}
\Bigl(
\langle (1-e^{\varepsilon^2 \tau \Delta} ) u^{n+1}, u^{n+1} \rangle
-\langle (1-e^{\varepsilon^2\tau\Delta} ) u^n, u^n \rangle
+ \langle( 1-e^{\varepsilon^2\tau \Delta} ) (u^{n+1}-u^n), u^{n+1} - u^n \rangle 
\Bigr). \label{eu_2.10}
\end{align}  
Multiplying \eqref{eu_2.7} by $(\tilde u^{n+1}-\tilde u^n)$, integrating over $\mathbb T^d$
and using \eqref{eu_2.9}--\eqref{eu_2.10}, we  obtain 
\begin{equation}\label{ineq:decay}
\begin{aligned}
& \widetilde E^{n+1} - \widetilde E^n  \\
& = -\frac1{2\tau}\left\langle \left(1- e^{\eps^2\tau \Delta} \right) (u^{n+1}- u^n), (u^{n+1}-
 u^n) \right\rangle-\left\langle\left(\frac1\tau-\frac12 \widetilde F''(\xi^n)\right) (\tilde u^{n+1}-\tilde u^n)^2,1\right\rangle.
\end{aligned}
\end{equation}
It is not difficult to check that
\begin{equation}
\widetilde F''(\xi) = \frac1\tau - \frac{e^\tau}{\tau\left(1+(e^{2\tau}-1) \xi^2\right)^{\frac32}} .
\end{equation}
Clearly
\begin{equation}
\frac1\tau -\frac12 F''(\xi)\ge 0, \qquad\forall\, \xi \in \mathbb R.
\end{equation}
Thus we have $\widetilde E^{n+1} \le \widetilde E^n$ for all $n\ge 0$.

\subsection{Uniform Sobolev bounds}
To establish uniform Sobolev bounds on $u^n$, we first show that it suffices to prove 
\begin{align} \label{eu_2.14}
\sup_{n\ge 1} \| \tilde u^n \|_{H^{k_0}} \le D_1,
\end{align}
where $D_1>0$ depends only on ($\eps$, $k_0$, $\|u^0 \|_{H^{k_0} }$, $d$).

Indeed assume \eqref{eu_2.14} holds, we now check the uniform bound on
$u^{n+1}$.  For simplicity we conduct the argument for $k_0=1$, i.e. we check 
the $H^1$ bound. The general $k_0\ge 2$ case is similar and omitted {(see e.g. the bootstrap argument developed in \cite[Section 2.7.2]{LT21}).}
Since $u^{n+1} = \mathcal S_{\mathcal L}\left( \tau/2\right) \mathcal S_{\mathcal N}\left( \tau \right)\mathcal S_{\mathcal L}\left( \tau/2\right) u^n$ and
$\tilde u^n = \mathcal S_{\mathcal L} \left(\tau/2\right) u^n$, we have
\begin{align}
u^{n+1} = \mathcal S_{\mathcal L}\left(\tau/2 \right) 
\mathcal S_{\mathcal N} (\tau) \tilde u^n.
\end{align}
By examining the structure of the equation
$\partial_t u = u-u^3$, it is not difficult to check that
\begin{align}
& \|  \nabla ( \mathcal S_{\mathcal N}(\tau ) a ) \|_2 \le e^{\tau} \| \nabla a \|_2,  \label{eu_2.16}\\
& \| \mathcal S_{\mathcal N}(\tau) a \|_2 \le \max\{ \| a\|_2, \; C_d\}, \label{eu_2.17}
\end{align}
where $C_d>0$ is a constant depending only on the dimension $d$.  We then discuss
two cases. If $0<\tau \le 1$,  we use \eqref{eu_2.16}--\eqref{eu_2.17} and obtain
\begin{align}
\| u^{n+1} \|_{H^1}  \le \| \mathcal S_{\mathcal N}(\tau) \tilde u^n \|_{H^1} \le D_2,
\end{align}
where $D_2>0$ depends only on ($\eps$, $\|u^0 \|_{H^1}$, $d$). 
If $\tau>1$,  we use \eqref{eu_2.17} and obtain
\begin{align}
\| u^{n+1} \|_{H^1}  \lesssim \| \mathcal S_{\mathcal N}(\tau) \tilde u^n \|_{2} \le D_3,
\end{align}
where $D_3>0$ depends only on ($\eps$, $\|u^0 \|_{H^1}$, $d$).  Thus in both cases we obtain
the uniform bound on $u^{n+1}$.

We now focus on \eqref{eu_2.14}. For simplicity we assume $k_0=1$. The general case $k_0\ge 2$
follows along similar lines using smoothing estimates and we omit the details.  

Consider first the case $0<\tau \le 1$.  We rewrite
\begin{align}
\widetilde F(z)
& = \frac 14+ \frac 1{2\tau} z^2 
- \frac {e^{\tau}} {\tau} \cdot \frac{ z^2} { 1+ \sqrt{ 1+(e^{2\tau} -1) z^2 }  } \notag \\
& = \frac 14+ \frac 1 {2\tau} z^2
-\frac 1 {\tau} \cdot 
\frac {z^2} {1+\sqrt{1+(e^{2\tau}-1) z^2} } 
- \frac{e^{\tau}-1}{\tau} \cdot \frac {z^2}{1+\sqrt{1+(e^{2\tau}-1) z^2} } \notag \\
&=  \frac 14+\frac{e^{2\tau}-1} {2\tau}
\cdot \frac{z^4} { \big(1+\sqrt{1+(e^{2\tau}-1) z^2} \big)^2}
-  \frac{e^{\tau}-1}{\tau} \cdot \frac {z^2}{1+\sqrt{1+(e^{2\tau}-1) z^2} } \notag \\
&= \frac{e^{2\tau}-1}{2\tau}
\Bigl( \big(A -\frac 1{1+e^{\tau} } \big)^2 - \frac 1 {(1+e^{\tau} )^2} \Bigr) +\frac 14,
\end{align}
where $A= z^2/\big(1+\sqrt{1+(e^{2\tau}-1) z^2}\big)$. 

Since $d\le 3$, by using the above expression together with Sobolev embedding, we have
\begin{align}
\int_{\mathbb T^d} \widetilde F(u^0 ) dx \le D_4,
\end{align}
where $D_4>0$ depends only on ($\|u^0\|_{H^1(\mathbb T^d)}$, $d$).   It follows
that  uniformly in $n$, 
\begin{align}
\frac 12 \eps^2 \| \nabla \tilde u^n \|_2^2 \le 
\frac 1 {2\tau}
\langle (e^{-\varepsilon^2 \tau \Delta} -1) \tilde u^n, \, \tilde u^n \rangle \le D_5,
\end{align}
where $D_5>0$ depends only on ($\|u^0 \|_{H^1(\mathbb T^d)}$, $d$, $\eps$).
By \eqref{eu_2.17}, it is not difficult to obtain uniform control of $\|\tilde u^n \|_2$. 
The desired uniform $H^1$ bound on $\tilde u^n$ follows easily. 

Next we consider the case $\tau>1$.  By \eqref{eu_2.8}, $ \int_{\mathbb T^d} \widetilde F(\tilde u^n) dx$ is clearly
controlled by $L^2$-norm of $\tilde u^n$. Since we have uniform control of
$\| \tilde u^n \|_2$, the desired uniform $H^1$ bound on $\tilde u^n$ follows easily.

\subsection{Connection with the standard energy}
Next, we show that the modified energy  coincides with the standard energy as the time step $\tau$ tends to $0$, i.e.
\begin{align}
\sup_{n\ge 0} | \widetilde E^n - E(u^n ) | \lesssim \tau, \label{Eu_2.23}
\end{align}
where the implied constant depends on ($\eps$, $d$, $u^0$).  Note here the working assumption
is $u^0 \in H^{20}(\mathbb T^d)$ and $0<\tau \le 1$. By the uniform Sobolev regularity result derived earlier, we have
uniform control of $H^{20}$-norm of $u^n$ for all $n\ge 0$.  Furthermore thanks to the uniform
Sobolev bound on $u^n$,  we only need
examine the regime $0<\tau \ll 1$. 

Firstly observe that
\begin{align}
\left| \frac 1 {2\tau}
\langle ( 1- e^{\eps^2 \tau \Delta} ) u^n, \, u^n \rangle
- \frac 12 \langle \eps^2 (-\Delta) u^n, u^n \rangle \right| \lesssim \tau.
\end{align}

Next for $0<\tau \ll 1$, we  have
\begin{equation}
\begin{aligned}
 \widetilde F(\tilde u^n)  & =\frac 14 +\frac1{2\tau} (\tilde u^n)^2 -
 \frac{e^\tau}{\tau(e^{2\tau}-1)}\left( \sqrt{1+(e^{2\tau}-1)(\tilde u^n)^2} -1\right) \\
&= \frac 14+\frac1{2\tau} (\tilde u^n)^2 - \frac{e^\tau}{\tau(e^{2\tau}-1)}\left( \frac12(e^{2\tau}-1) (\tilde u^n)^2 -\frac18(e^{2\tau}-1)^2 (\tilde u^n)^4 +\mathcal O((e^{2\tau}-1)^3)\right)\\
& = \frac 14-\frac{e^\tau-1}{2\tau} (\tilde u^n)^2 + \frac{1}{8\tau}e^\tau(e^{2\tau}-1) (\tilde u^n)^4+\mathcal O(\tau^{-1}(e^{2\tau}-1)^2 )\\
& =\frac 14 -\frac12 (\tilde u^n)^2 + \frac14 (\tilde u^n)^4 +\mathcal O(\tau).
\end{aligned}
\end{equation}
Since $u^n$ and $\tilde u^n$ differ by $\mathcal O(\tau)$,  we obtain
\begin{align}
&\left|\widetilde F(\tilde u^n) - \frac 14 ((u^n)^2-1)^2 \right| \lesssim \tau; \\
&\left|\int_{\mathbb T^d} \widetilde F(\tilde u^n)  dx - 
\int_{\mathbb T^d} \frac 14 ((u^n)^2-1)^2 dx\right| \lesssim \tau.
\end{align}
Thus \eqref{Eu_2.23} is shown.

\subsection{Uniform second-order approximation}
For convenience of notation, we denote
\begin{align}
L = \eps^2 \Delta.
\end{align}
We first check the consistency for the propagator $
\mathcal S_{\mathcal L}(\frac{\tau}2)\mathcal S_{\mathcal N} (\tau) 
\mathcal S_{\mathcal L}(\frac {\tau}2)$. 
Concerning the operator $\mathcal S_{\mathcal N}(\tau)$, we  note that if
\begin{align}
\begin{cases}
\partial_t w = w-w^3, \quad 0<t\le \tau; \\
w\Bigr|_{t=0} = b,
\end{cases}
\end{align}
where $w$ admits uniform control of its Sobolev norm, then 
\begin{align}
w(\tau) = b + \tau (b-b^3) + \frac 12 \tau^2
(1-3b^2)(b-b^3) + \mathcal O(\tau^3).
\end{align}
If $b= \mathcal S_{\mathcal L}(\frac{\tau}2) a = a + \frac{\tau}2 L a
+ \frac{\tau^2} 8 L^2 a +\mathcal O (\tau^3)$, then we can simplify the above
further and obtain
\begin{align}
w(\tau)&=a+\frac{\tau}2 L a + \frac{\tau^2}8 L^2 a
+ \tau \Bigl( (a+\frac{\tau}2 La) - (a+ \frac{\tau}2 L a)^3 \Bigr)
+ \frac{\tau^2}2 (1-3a^2)(a-a^3) + \mathcal O(\tau^3) \notag \\
& = a+ \tau( a-a^3+\frac 12 La)+
\tau^2 \Bigl( \frac 18 L^2 a +\frac 12 La -\frac 32 a^2  La
+ \frac 12 (1-3a^2) (a-a^3) \Bigr)+ \mathcal O(\tau^3).
\end{align}
Now if $u= \mathcal S_{\mathcal L}(\frac{\tau}2)\mathcal S_{\mathcal N} (\tau) 
\mathcal S_{\mathcal L}(\frac {\tau}2) a$, we have 
\begin{align}
u &=\mathcal S_{\mathcal L}(\frac{\tau}2) w(\tau) + \mathcal O(\tau^3) \notag \\
& =  \mathcal S_{\mathcal L}(\frac {\tau}2) (a+\frac 12 \tau  La + \frac 18 \tau^2 L^2 a) \notag \\
& \quad + \mathcal S_{\mathcal L}(\frac {\tau}2)
\left( \tau(a-a^3) +\tau^2 \Bigl(  \frac 12 La -\frac 32a^2  La
+ \frac 12 (1-3a^2) (a-a^3) \Bigr) \right)+ \mathcal O(\tau^3) \notag \\
&= \mathcal S_{\mathcal L}(\tau) a+ \tau (a-a^3)
+ \frac{\tau^2}2
\Bigl( L(a-a^3)+ La -3a^2 La+(1-3a^2)(a-a^3) \Bigr) 
+ \mathcal O(\tau^3). \label{Eu_2.32}
\end{align}
We now turn to the expansion  of the exact PDE solution. 
Let $u^{\mathrm{P}}$ be the exact PDE solution to \eqref{1.1} with initial data $\tilde a$.  We have
\begin{align}
u^{\mathrm{P}} (\tau) & = \mathcal S_{\mathcal L}(\tau) \tilde a + \int_0^{\tau} 
\mathcal S_{\mathcal  L}(\tau -s) (u(s)-u(s)^3) ds \notag \\
& =\mathcal S_{\mathcal L}(\tau)\tilde a + \int_0^{\tau} ( 1+ (\tau-s) L) (u(s)-u(s)^3) ds + \mathcal O(\tau^3) \notag \\
& = \mathcal S_{\mathcal L}(\tau)\tilde a+ \int_0^{\tau} 
\Bigl(  \tilde a -\tilde a^3 + s( 1-3\tilde a^2)  (L\tilde a +\tilde a -\tilde a^3) 
\Bigr) ds + 
\int_0^{\tau} (\tau-s)L (\tilde a-{\tilde a}^3) ds + \mathcal O(\tau^3) \notag \\
& = \mathcal S_{\mathcal L}(\tau)\tilde a+ \tau (\tilde a-{\tilde a}^3)
+ \frac {\tau^2}2 \Bigl(  L(\tilde a-{\tilde a}^3) +(1- 3{\tilde a}^2 )(L\tilde a+ \tilde a-{\tilde a}^3) \Bigr)+ \mathcal O(\tau^3). \label{Eu_2.33}
\end{align}
Clearly  \eqref{Eu_2.32} and \eqref{Eu_2.33} have the same form in  $\mathcal O(\tau^3)$.

{
Albeit standard, we now outline how to obtain the global error estimate $\mathcal O(\tau^2)$ for $n\tau\le T$. 
Denote $\mathcal T(\tau)$ as the solution operator $u(0)\mapsto u(\tau)$ to the exact PDE problem:
\begin{align} \label{1.1}
\begin{cases}
\partial_t u  = \varepsilon^2 \Delta  u-  f(u),\\
u\Bigr|_{t=0} =u(0).
\end{cases}
\end{align}
Since we assume high regularity on the initial data $u^0$ (see the description before \eqref{Eu_1.23}), we have 
\begin{equation}\label{J18_22a}
\sup_{n \tau\le T }
\left( \| u^n \|_{H^{k} (\mathbb T^d) }+\| u(n\tau) \|_{H^{k} (\mathbb T^d) }\right)\le C,
\end{equation}
where $u(n\tau)=\mathcal T(n\tau)u^0$ corresponds to the exact PDE solution. 
Now we write
\begin{equation}
\left\{
\begin{aligned}
&u^n = \underbrace{\mathcal S_{\mathcal L}(\frac{\tau}2)\mathcal S_{\mathcal N} (\tau) 
\mathcal S_{\mathcal L}(\frac {\tau}2)}_{=: \mathcal S(\tau)} u^{n-1},\\
&u(n\tau) = \mathcal T(\tau) u((n-1)\tau).
\end{aligned}
\right.
\end{equation} 
Clearly by the triangle inequality, we have
\begin{equation}
\|u^n-u(n\tau)\|_2\le 
\|\mathcal S(\tau) u^{n-1}-\mathcal T(\tau) u^{n-1}\|_2
+ \|\mathcal T(\tau) u^{n-1}-\mathcal T(\tau) u((n-1)\tau)\|_2.
\end{equation}
By using \eqref{Eu_2.32}, \eqref{Eu_2.33} and \eqref{J18_22a}, we have
\begin{equation}
\|\mathcal S(\tau) u^{n-1}-\mathcal T(\tau) u^{n-1}\|_2\le B_1\tau^3,
\end{equation} 
where $B_1>0$ is independent of $\tau$. 
By stability of the exact PDE solution and \eqref{J18_22a}, we have
\begin{equation}
\|\mathcal T(\tau) u^{n-1}-\mathcal T(\tau) u((n-1)\tau)\|_2\le  e^{B_2\tau} \|u^{n-1}-u((n-1)\tau)\|_2,
\end{equation}
where $B_2>0$ is independent of $\tau$. 
It follows that 
\begin{equation}
\|u^n-u(n\tau)\|_2 \le e^{B_2\tau} \|u^{n-1}-u((n-1)\tau)\|_2 + B_1\tau^3. 
\end{equation}
An elementary analysis gives
\begin{equation}
\sup_{n\tau\le T} \|u^n-u(n\tau)\|_2 \le \mathcal O(\tau^2). 
\end{equation}
}

\section{The case with logarithmic potentials}\label{sect3}
In this section, we consider the Allen--Cahn equation with logarithmic potential, i.e.
\begin{equation}\label{Eu_3.1}
\partial_t u = \varepsilon^2 \Delta u + \theta_c u -\frac{\theta}{2} \bigl( \ln(1+u) - \ln(1-u) \bigr),
\end{equation}
where $0<\theta<\theta_c$. We shall consider Strang-type second order in time splitting.
Define $\mathcal S_{\mathcal L}(\tau) =
\exp(\eps^2 \tau \Delta)$. 
For the nonlinear propagator, it is natural to consider the equation
\begin{align}\label{eq3-3}
\begin{cases}
\partial_t w = \theta_c w-\theta\artanh(w), \quad 0<t\le \tau; \\
w\Bigr|_{t=0}=w_0.
\end{cases}
\end{align}
Here 
\begin{equation}
\artanh(w) = \frac{1}{2} \bigl( \ln(1+w) - \ln(1-w) \bigr).
\end{equation}
is the inverse hyperbolic function.  Define $\mathcal S_{\mathcal N}^{(\mathrm{LOG})} (\tau)$
as the nonlinear solution operator $w_0 \rightarrow w(\tau)$.  Theoretically speaking, 
one can develop the stability theory
for the Strang-splitting approximation
\begin{align}
u^{n+1} = \mathcal S_{\mathcal L} (\frac {\tau}2)
\mathcal S_{\mathcal N}^{(\mathrm{LOG})} (\tau) 
\mathcal S_{\mathcal L}(\frac {\tau}2) u^n.
\end{align}
However on the practical side there is a serious issue.  Namely in stark contrast
to the polynomial case, the system \eqref{eq3-3} does not admit an explicit 
solution formula. In yet other words, the solution operator $\mathcal S_{\mathcal N}^{(\mathrm{LOG})} (\tau) $ is difficult to implement in practice unless one makes a further
discretization or approximation.  As we shall see momentarily, we shall resolve this problem
by approximating $\mathcal S_{\mathcal N}^{(\mathrm{LOG})} (\tau) $ via a judiciously 
chosen numerical discretization.  We should point it out that the choice of the numerical solver is a
rather subtle and technically involved one, since there are at least two issues to keep in mind for the construction
of the numerical solver:
\begin{enumerate}
\item \underline{$\mathcal O(\tau^3)$-truncation error}. This is to ensure the genuine \emph{Strang}-nature of the scheme. Since the Strang-splitting is a second order in time scheme,  the truncation error must be kept within $\mathcal O(\tau^3)$ for the numerical solver.
\item \underline{Strict phase separation}. The numerical solver needs to preserve a sort of 
maximum principle of the form $|u|\le u_*<1$ to ensure strict phase separation and stability
of the overall scheme.
\end{enumerate}
In what follows we shall define $g(u)$ as
\begin{equation}\label{eq:fu}
g(u) = \theta_c u - \theta\artanh(u).
\end{equation}
The condition $0<\theta<\theta_c$ is always in force. Note that 
\begin{align}
& g^{\prime}(u) = \theta_c - \theta \frac 1 {1-u^2};\\
& g^{\prime\prime}(u) = - \theta  \frac {2u}{(1-u^2)^2}. \label{Eu_3.7}
\end{align}
In particular $g$ is concave on the interval $(0,1)$. Since $g(0)=0$ and $g(1-)=-\infty$,
by using concavity it is not difficult to check that $g$ admits
a unique root in $(0,1)$. Thereby we denote by $u_*$  this unique solution of $g(u) = 0$ in $(0,1)$. 
One can see the left plot of Figure \ref{fig:root} for an example of the profile of $g$.
It is not difficult to check that if $u$ is a smooth solution to \eqref{Eu_3.1} 
satisfying $\| u\|_{\infty} \le u_*$ initially at time zero, then
\begin{align}
\sup_{t>0} \left\| u(t,\cdot)\right\|_\infty \le u_*.
\end{align}
In designing the numerical solver it is of pivotal importance to preserve the maximum principle.

\begin{figure}[!h]
\includegraphics[width = 0.49\textwidth]{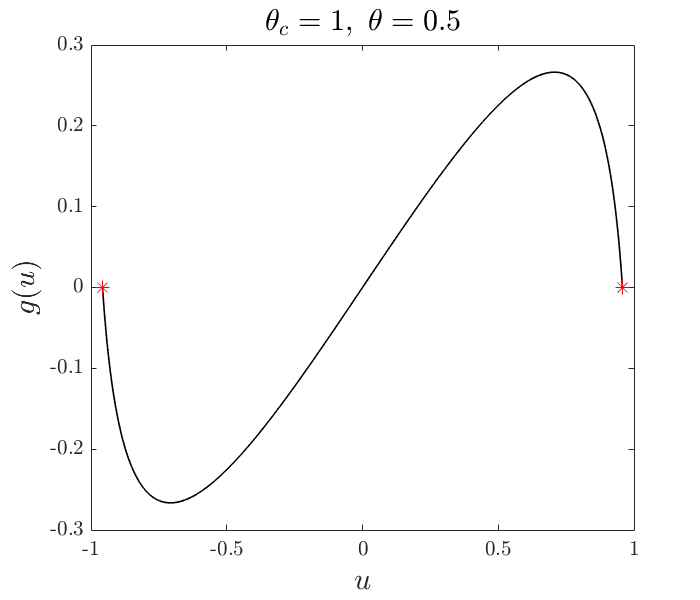}
\includegraphics[width = 0.49\textwidth]{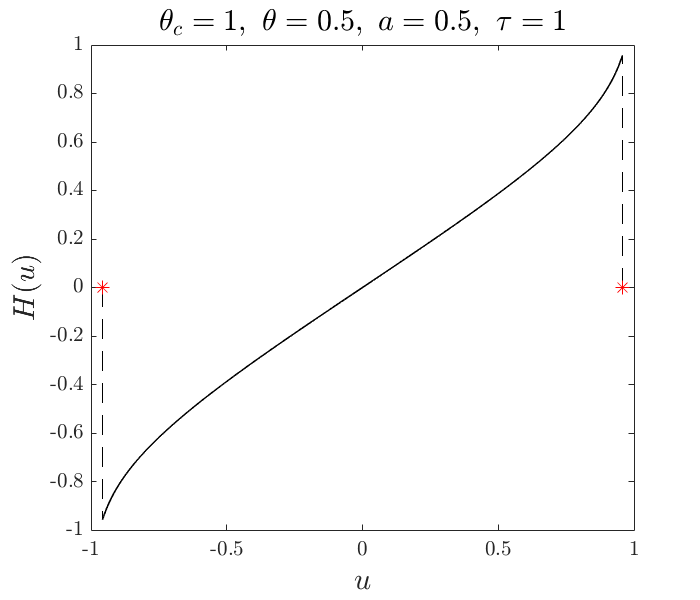}
\caption{$g(u)$ (left) and $H(u)$ (right) w.r.t. $u\in [-u_*,u_*]$, where $\theta_c =1$,  $\theta= \frac12$, and $a = \frac12$. 
The red star markers denote the nonzero roots of $g$, i.e., $-u_*$ and $u_*$. }\label{fig:root}
\end{figure}

\subsection{PR-RK method for approximating $\mathcal S_{\mathcal N}^{(\mathrm{LOG})}(\tau)$}
Diagonally Implicit Runge-Kutta (DIRK) formulae have been widely used for the numerical solution of stiff initial value problems. The simplest method from this class is the second order implicit midpoint method.  To approximate $\mathcal S_{\mathcal N}^{(\mathrm{LOG})}(\tau)v$ for a given function $v$, we shall use the Pareschi and Russo's two-stage diagonally implicit Runge Kutta (PR-RK) method \cite{PR05}  (see Table \ref{tab:PR}) . One should note that under the assumption of uniform Sobolev bounds on the
numerical iterates, the  truncation error involved is guaranteed to be within $\mathcal O(\tau^3)$. 
 \begin{table}[htbp]
        \centering
        \renewcommand\arraystretch{1.75}
       \begin{tabular}{c|cc}
         $a$ & $a$ & $0$\\
         $1-a$& $1-2a$ & $a$\\
         \hline
         & $\frac 12$ & $\frac 12$\\
       \end{tabular}
       \vspace{0.1in}
       \caption{Butcher tableau of Pareschi and Russo's Runge-Kutta method with constant $a$.}\label{tab:PR}
\end{table}

More precisely, for given $v$ we approximate $\mathcal S_{\mathcal N}^{(\mathrm{LOG})}(\tau) v$ via two internal stages (below $a$ is a constant parameter):
\begin{equation}\label{eq:PRRK}
 \begin{aligned}
& u_1 = v+ a\tau g(u_1),\\
& u_2 = v+(1-2a)\tau g(u_1)+a \tau g(u_2),\\
& \mathcal S_{\mathcal N}^{(\mathrm{LOG})}(\tau) v
\approx \widetilde{\mathcal S}_{\mathcal N}(\tau)  v : = v+ \frac{1}{2}\tau g(u_1) + \frac{1}{2}\tau g(u_2).
\end{aligned}
\end{equation}
Since the PR-RK method is a second order method,  it is not difficult to check that if
\begin{align} \label{Eu_3.10}
\max\{ \| v\|_{\infty}, \, \| u_1\|_{\infty}, \, \| u_2\|_{\infty},
\|\widetilde{\mathcal S}_{\mathcal N}(\tau) v \|_{\infty} \} \le u_*<1,
\end{align}
and $v$ has uniform Sobolev bounds, then
\begin{equation}
\mathcal S_{\mathcal N}^{(\mathrm{LOG})}(\tau) v= \widetilde{\mathcal S}_{\mathcal N}(\tau)  v +
\mathcal O(\tau^3).
\end{equation}
We shall verify \eqref{Eu_3.10} later under certain parametric conditions on ($a$, $\theta$, $\theta_c$, $\tau$). 

Concluding from the above discussion, we are led to the following RK based Strang-type splitting
algorithm for \eqref{Eu_3.1}:
\begin{equation} \label{Eu_3.12}
\boxed{
\begin{aligned}
&\qquad u^{n+1} = \mathcal S_{\mathcal L}(\tau/2)  \widetilde{\mathcal S}_{\mathcal N}(\tau) \mathcal S_{\mathcal L}(\tau/2) u^n;  \\
&\text{In terms of  $ \tilde u^n = \mathcal S_{\mathcal L}(\tau/2) u^n$, we have equivalently:}
 \\
&\qquad \tilde u^{n+1} = \mathcal S_{\mathcal L}(\tau) 
 \widetilde{\mathcal S}_{\mathcal N}(\tau) \tilde u^n, 
\end{aligned}}
\end{equation}
where $\widetilde{\mathcal S}_{\mathcal N}(\tau)$ is defined via \eqref{eq:PRRK}.  We tacitly
assume that $\widetilde{\mathcal S}_{\mathcal N}(\tau)$ is the exact solver of
\eqref{eq:PRRK} and do not consider other intermediate numerical errors due to the implicit nature of the
scheme.   The validity of this assumption will be examined in the next subsection.

\subsection{Solvability of \eqref{eq:PRRK}}
Although the first two equations of \eqref{eq:PRRK} are implicit, they can be tackled by the Newton method efficiently with quadratic convergence. In practice only a few iterations are needed
to achieve machine precision. The first two equations in \eqref{eq:PRRK} can be rewritten as
\begin{align}
& H(u_1) -v= 0; \label{eq:H1} \\
& H(u_2) -v -(1-2a)\tau g(u_1) = 0, \label{eq:H2}
\end{align}
where 
\begin{equation}
H(u) \coloneqq u-a\tau g(u) = \left(1-a\tau \theta_c\right) u+a\tau \theta \artanh(u).
\end{equation}
See the right plot of Figure \ref{fig:root} for a graphical illustration of $H(u)$.

To solve \eqref{eq:H1}, we implement the Newton iteration
\begin{align} \label{eq:newton1}
\begin{cases}
u_1^{(k+1)} = u_1^{(k)} - \frac{H(u_1^{(k)})-v}{H'(u_1^{(k)})},\quad k\ge 0;
  \\
u_1^{(0)} = \mathrm{sign}(v) u_*.
\end{cases}
\end{align}
Similarly, we use the following Newton iteration to solve  \eqref{eq:H2}
\begin{align}\label{eq:newton2}
\begin{cases}
u_2^{(k+1)} = u_2^{(k)} - \frac{H(u_2^{(k)})-v-(1-2a)\tau g(u_1) }{H'(u_2^{(k)})}, 
\quad k\ge 0; \\
u_2^{(0)} = \mathrm{sign}(v) u_*.
\end{cases}
\end{align}
\begin{lem}[Unique solvability \& convergence of Newton iterations]\label{lem:conv2}
Assume that $|v|\le u_*$. 
If $0<\tau\le \frac{1}{(3a-1)(\theta_c-\theta)}$ with $a\ge \frac12$, then \eqref{eq:H1} and  \eqref{eq:H2} are uniquely solvable,  and the Newton iterations \eqref{eq:newton1} and \eqref{eq:newton2} converge.
\end{lem}
\begin{proof}

Without loss of generality, we consider the case when $0<v< u_*$.
Direct computation gives
\begin{equation}
\begin{aligned}
&H'(u) = 1 - a\tau\theta_c+ \frac{a\tau\theta}{1-u^2},\quad H''(u) = \frac{2a\tau\theta u}{(1-u^2)^2}.
\end{aligned}
\end{equation}
From the condition $0<\tau\le \frac{1}{(3a-1)(\theta_c-\theta)}$ with $a\ge \frac12$, we have 
$0<\tau \le \frac 1 {a(\theta_c-\theta)}$.  Thus
\begin{align}
H^{\prime}(u) >1- a \tau \theta_c + a\tau \theta  \ge 0, \quad \forall\, 0<u\le u_*.\quad \Rightarrow
\quad H^{\prime}(u)>0, \qquad\forall\, 0<u \le u_*.
\end{align}
It is also clear that $H''(u)>0$ for any $0<u\le u_*$. 

For  \eqref{eq:H1},  using the fact that $f(u_*)=0$ we have
\begin{equation}
\begin{aligned}
& H(u_*)-v= u_*-v > 0, \\
& H(v) -v= -a\tau g(v) < 0.
\end{aligned}
\end{equation}
Therefore, \eqref{eq:H1} is uniquely solvable. 
Given $u_1^{(0)} = u_*$, it follows that the Newton iteration \eqref{eq:newton1} converges to the unique root $u_1$ satisfying $v< u_1 < u_* $.

We turn now to \eqref{eq:H2}.  By using the fact that $g^{\prime}(u) \le \theta_c -\theta$
for $0\le u <1$, we have $g(u) \le (\theta_c-\theta)u$ for any $0<u <u_*$.  By using
\begin{equation}
\begin{aligned}
& \tau{(3a-1)(\theta_c-\theta)} \le 1,\\
& g(u)\le (\theta_c-\theta) u,\quad\forall \,0<u<u_*,
\end{aligned}
\end{equation}
we have 
\begin{equation}
\begin{aligned}
& H(0) -v -(1-2a)\tau g(u_1) = -u_1+(3a-1)\tau g(u_1) \le 0, \\
& H(u_*)-v -(1-2a)\tau g(u_1) = u_*-v+(2a-1)\tau g(u_1)  > 0.
\end{aligned}
\end{equation}
Therefore, \eqref{eq:H2} is uniquely solvable. 
Given $u_2^{(0)} = u_*$, it follows that the Newton iteration \eqref{eq:newton1} converges to some root $0 < u_2 < u_*. $
\end{proof}
\subsection{The maximum principle}
In this subsection, we show that the RK-based Strang-splitting method
\eqref{Eu_3.12} preserves the maximum principle.  

\begin{thm}[Maximum principle]\label{thm:max}
Denote by $u_*$  the unique root of $g(u)=0$ in $(0,1)$. 
If $a\ge 1+\frac{\sqrt{2}}{2}$, $0<\tau\le \frac{1}{(3a-1)(\theta_c-\theta)}$, and $\|\tilde u^n\|_\infty\le u_{*}$, then 
\begin{align}
\| \widetilde{\mathcal S}_{\mathcal N}(\tau) \tilde u^n \|_{\infty} \le u_*,
\end{align}
where $\widetilde{\mathcal S}_{\mathcal N}(\tau)$ was defined in \eqref{eq:PRRK}.
It follows that
\begin{align}
&\|\tilde u^{n+1}\|_\infty = \| \mathcal S_{\mathcal L}(\tau)  \widetilde{\mathcal S}_{\mathcal N}(\tau) \tilde u^n\|_{\infty} \le u_{*}; \\
& \| u^{n+1} \|_{\infty} = \| \mathcal S_{\mathcal L}(\tau/2)  \widetilde{\mathcal S}_{\mathcal N}(\tau) \tilde u^n \|_{\infty} \le u_*.
\end{align}
\end{thm}

\begin{proof}
It suffices for us to treat $\tilde u^n$ as a real number. 
With no loss we  consider the case when $0<\tilde u^n<u_*$. 
As $a\ge 1+\frac{\sqrt{2}}{2}$, we have shown in the proofs of Lemma \ref{lem:conv2} that  
\begin{equation}
\tilde u^n < u_1< u_*\quad\mbox{and}\quad 0<u_2<u_*.
\end{equation}
By \eqref{eq:PRRK}, we have
\begin{align}
 \tau g(u_1) &= \frac 1a (u_1 -\tilde u^n); \\
 \tau g(u_2)  &= \frac 1a ( u_2 -\tilde u^n ) - \frac 1a (1-2a) \tau g(u_1)
= \frac 1a (u_2- \tilde u^n ) -\frac 1{a^2} (1-2a)(u_1-\tilde u^n) \notag \\
&=(\frac 1 {a^2} -\frac 3a) \tilde u^n + \frac 1 a u_2 +\frac{2a-1}{a^2} u_1.
\end{align}
Since $a\ge 1+\frac{\sqrt{2}}{2}$,  it is not difficult to check that the following inequality holds
\begin{equation}
\frac{3}{2a}-\frac{1}{2a^2}\ge 0\quad\mbox{and}\quad 1-\frac{2}{a}+\frac{1}{2a^2}\ge 0.
\end{equation}
Consequently we have
\begin{equation}
\begin{aligned}
& \widetilde{\mathcal S}_{\mathcal N}(\tau)  \tilde u^n   = \tilde u^n + \frac{1}{2}\tau g(u_1) + \frac{1}{2}\tau g(u_2) \\
& \qquad \quad=  \frac1{2a} u_2 + \left( \frac{3}{2a}-\frac{1}{2a^2}\right) u_1 + \left(1-\frac{2}{a}+\frac{1}{2a^2}\right) \tilde u^n< u_*,\\
& \widetilde{\mathcal S}_{\mathcal N}(\tau)  \tilde u^n   = \tilde u^n  + \frac{1}{2}\tau g(u_1) + \frac{1}{2}\tau g(u_2) > \tilde u^n.
\end{aligned}
\end{equation}
It follows that 
$\|\widetilde{\mathcal S}_{\mathcal N}(\tau)  \tilde u^n\|_\infty < u_*.$
More generally, if $\|\tilde u^n\|_\infty\le u_{*}$, then
\begin{equation}
\|  \widetilde{\mathcal S}_{\mathcal N}(\tau) \tilde u^n \|_\infty \le u_*.
\end{equation}
Thus
\begin{equation}
\|\tilde u^{n+1} \|_\infty = \|\mathcal S_{\mathcal L}(\tau)  \widetilde{\mathcal S}_{\mathcal N}(\tau) \tilde u^n\|_\infty \le \|\widetilde{\mathcal S}_{\mathcal N}(\tau) \tilde u^n\|_\infty \le u_*.
\end{equation}
The bound for $u^{n+1}$ follows similarly.
\end{proof}

\subsection{Modified energy dissipation}
By \eqref{Eu_3.12}, we  have
\begin{equation}
e^{-\eps^2\tau\Delta}\tilde u^{n+1} = \widetilde{\mathcal S}_{\mathcal N}(\tau)\tilde u^n.
\end{equation}
Clearly 
\begin{equation}\label{Eu_3.34}
\frac 1{\tau}\left(e^{-\eps^2\tau\Delta}-1\right)\tilde u^{n+1}  + \frac 1 {\tau} (\tilde u^{n+1}-\tilde u^n)=\frac 1{\tau} \Bigl(  \widetilde{\mathcal S}_{\mathcal N}(\tau)\tilde u^n -\tilde u^n \Bigr).
\end{equation}
By \eqref{eq:PRRK}, we have
\begin{equation}
\frac 1{\tau} \Bigl( \widetilde{\mathcal S}_{\mathcal N}(\tau)\tilde u^n -\tilde u^n \Bigr)=  \frac{1}{2} g(u_1) + \frac{1}{2} g(u_2).
\end{equation}
The strategy is to rewrite the RHS above as  $-\overline{F}^{\prime}(\tilde u^n)$, where 
$\overline{F}$ is a one-variable function serving as the potential energy function. 
For this we need to introduce some notation.

Recall that in \eqref{eq:PRRK}, $u_1$ and $u_2$ are implicitly defined as a function of $u$
for given $u$.  For convenience of notation, 
we regard $u_1=u_1(u)$, $u_2=u_2(u) $ as two smooth functions of $u\in [-u_*, u_*]$ solving
\begin{equation}\label{eq:u12}
u_1(u) =  u+ a\tau g(u_1 (u) ),\quad
u_2(u) = u + (1-2a) \tau g(u_1 (u) )+a\tau g(u_2(u) ).
\end{equation}
We define  $\overline F =\overline{F}(u)$ as the unique smooth function satisfying 
\begin{equation}\label{eq:F_log}
\overline F'(u) = \frac d{du} \overline{F}(u)= -\frac12 g(u_1(u) ) - \frac{1}{2} g(u_2 (u) ),\quad\mbox{with } \overline F(0) = 0.
\end{equation}
The normalization $\overline F(0)=0$ is chosen in analogy with \eqref{eq:logpot} since 
$F_{\mathrm{fh}}(0)=0$.  With the help of $\overline F$, we rewrite  \eqref{Eu_3.34}
 as
\begin{equation}\label{Eu_3.38}
\frac 1{\tau}\left(e^{-\eps^2\tau\Delta}-1\right)\tilde u^{n+1}  + \frac 1 {\tau} (\tilde u^{n+1}-\tilde u^n)=- \overline F^{\prime} ( \tilde u^n ).
\end{equation}

\begin{thm}[Modified energy dissipation]\label{thm:ene2}
Assume $u^0 \in H^1(\mathbb T^d)$ and  $\| u^0 \|_{\infty} \le u_*$ where $u_*$ is the unique root of $g(u)=0$ in
$(0,1)$. 
When $ a\ge 1+\frac{\sqrt 2}{2}$ and $0<\tau\le \frac{1}{3a(\theta_c-\theta)}$, 
 the RK-based Strang splitting method \eqref{Eu_3.12} preserves the maximum principle and
 the modified energy dissipation property, namely
\begin{align} 
& \sup_{n\ge 1} \max\{ \| u^n \|_{\infty}, \; \| \tilde u^n \|_{\infty} \} \le u_*; \\
 &\overline E^{n+1} \le \overline E^n,\quad \forall\; n\ge 1;  \label{Eu_3.39}\\
&\overline E^n \coloneqq \frac1{2}\left\langle \frac{1}{\tau}(e^{-\eps^2\tau \Delta} -1) \tilde u^n, \tilde u^n \right\rangle + \left\langle\overline F(\tilde u^n),1\right\rangle,
\end{align}
where  $\tilde u^n  = \mathcal S_{\mathcal L}(\tau/2) u^{n}$ and $\overline F$ is defined by \eqref{eq:F_log}.
\end{thm}
\begin{rem}
For the energy dissipation to hold, formally speaking the argument only requires the weaker
condition $a\ge \frac 1 2$ and $0<\tau\le \frac{1}{3a(\theta_c-\theta)}$. However in order
to have solvability of our RK-based scheme for all $n\ge 1$, we need to impose the stronger condition on the parameter
$a$ in order to preserve the maximum principle.
\end{rem}
\begin{proof}
We only need to show  \eqref{Eu_3.39}.
Direct computation gives
\begin{equation}\label{eq:fs}
g'(u)= \theta_c - \frac{\theta}{1-u^2} \le \theta_c-\theta, \quad \forall u\in [-u_*, u_*].
\end{equation}
By \eqref{eq:F_log}, we have
\begin{equation}\label{eq:F_log2}
\overline F''(u) = -\frac12 g'(u_1) u_1'(u) - \frac{1}{2} g'(u_2) u_2'(u), \quad \forall u\in [-u_*,u_*].
\end{equation}
Taking the derivative of two equations in \eqref{eq:u12} w.r.t. $u$, we have
\begin{equation}\label{eq:u12'}
\begin{aligned}
& u_1'(u) = \frac{1}{1-a\tau g'(u_1)},\\
& u_2'(u) = \frac{(3-\frac1 a)-(2-\frac1 a)u_1'(u)}{1-a\tau g'(u_2)}.
\end{aligned}
\end{equation}
By \eqref{eq:fs} and the assumption $0<\tau\le \frac{1}{3 a (\theta_c-\theta)}$, we have 
 \begin{equation}\label{eq:f'u}
g'(u)\le \theta_c-\theta  \le \frac{1}{3 a\tau},  \quad \forall\, u\in [-u_*, u_*].
 \end{equation}
This yields
\begin{equation}\label{ineq:u1'}
 0<u_1'(u)\le \frac3 2\quad \mbox{and}\quad \left(3-\frac1 a\right)-\left(2-\frac1 a\right)u_1'(u)> 0. 
 \end{equation}
 Substituting \eqref{eq:u12'} into \eqref{eq:F_log2},  we then have
 \begin{equation}\label{eq:F''}
 \begin{aligned}
 \overline F''(u) &= -\frac{1}{\frac{2}{g'(u_1)}-2a\tau} -  \frac{(3-\frac1 a)-(2-\frac1 a)u_1'(u)}{\frac{2}{g'(u_2)}-2 a\tau}\\
 & \le \frac{1}{2a\tau}+ \frac{3-\frac 1 a}{2a\tau} 
= \frac{4-\frac 1 a}{2a\tau} \le \frac2\tau,\qquad\forall u\in [-u_*,u_*].
 \end{aligned}
 \end{equation}

Multiplying \eqref{Eu_3.38}  with $(\tilde u^{n+1}-\tilde u^n)$ and integrating over $\mathbb T^d$, we  obtain 
\begin{equation}\label{ineq:decay}
\begin{aligned}
& \overline E^{n+1} - \overline E^n  \le  -\left\langle\left(\frac1\tau-\frac12 \overline F''(\xi^n)\right) (\tilde u^{n+1}-\tilde u^n)^2,1\right\rangle \le 0,
\end{aligned}
\end{equation}
where $-u_*\le \xi^n\le u_*$ is some function between $\tilde u^n$ and $\tilde u^{n+1}$.
\end{proof}

\begin{rem}
The restrictions on ($a$, $\tau$) in Theorem \ref{thm:max} and \ref{thm:ene2} do not depend on $\sup_{|u| \le u_*} |g'(u)| $ which could be very large. 
\end{rem}

We now complete the proof of Theorem \ref{thm2}.
\begin{proof}[Proof of Theorem \ref{thm2}]
The first two statements follow from Theorem \ref{thm:ene2}. 
The rest of the statements can be proved along similar lines as in Theorem 
\ref{thm1}. We omit the details.
\end{proof}

\section{Numerical results}\label{sect4}
In this section, we implement the Strang splitting methods on the AC equation \eqref{1.1} and \eqref{Eu_3.1} with periodic boundary conditions. 

{\bf Space discretization. }
We use the spectral method to compute the linear solution operator $\mathcal S_{\mathcal L}(\tau)$ in the Strang splitting method.
Suppose that $\Omega = [0, L]^2$ with $L>0$ is a periodic torus. 
We can compute $e^{t\Delta_h} u$ via fast Fourier transform (FFT) as follows. 
Denote by $\Delta_h$ the discrete Laplacian operator with $h= \frac L N$ and $N\ge 1$ being an integer.
We use the following convention for FFT: 
\begin{align}
& u_{\mathbf j} = \sum_{-\frac N2 < k_x,k_y \le \frac N2} \hat u_{\mathbf k} \,e^{\frac{2\pi i}{N} \mathbf j\cdot\mathbf k }, \\
& \hat u_{\mathbf k} = \frac 1 {N^2} \sum_{j_x,j_y=0}^{N-1} u_{\mathbf j} \, e^{-\frac{2\pi i}{N} \mathbf j\cdot\mathbf k},
\end{align}
where $\mathbf j = (j_x,j_y)$ and $\mathbf k = (k_x,k_y)$. 
 We have
\begin{align}
(\Delta_h u)_{\mathbf j} = \frac {u_{j_x,j_y+1} + u_{j_x,j_y-1}+u_{j_x+1,j_y}+u_{j_x-1,j_y} -4 u_{j_x,j_y} } {h^2}.
\end{align}
Clearly
\begin{align}
(\widehat{\Delta_h u})_{\mathbf k} = w_{\mathbf k} \hat u_{\mathbf k},\quad -\frac N2 < k_x, k_y \le \frac N2,
\end{align}
where
\begin{align}
w_{\mathbf k} = h^{-2} \left( 2 \cos\left(2\pi \frac {k_x} N\right) +2
\cos\left(2\pi \frac {k_y} N\right) -4\right).
\end{align}
It follows that
\begin{align}
\widehat{(\mathcal S_{\mathcal L}(\tau) u)}_{\mathbf k}  = \widehat{(e^{\eps^2\tau\Delta_h}  u)}_{\mathbf k} = e^{\eps^2\tau w_{\mathbf k} } \hat u_{\mathbf k},\quad -\frac N2 < k_x, k_y \le \frac N2.
\end{align}
Taking the inverse fast Fourier transform (IFFT) then produces the numerical values of $\mathcal S_{\mathcal L}(\tau) u$ on the real side.

Similarly we compute $e^{-\eps^2\tau \Delta} u$ in the definition of modified energy  by taking the IFFT of the following equation
\begin{equation}
\widehat{(e^{-\eps^2\tau\Delta_h}  u)}_{\mathbf k} = e^{-\eps^2\tau w_{\mathbf k} } \hat u_{\mathbf k},\quad -\frac N2 < k_x, k_y \le \frac N2.
\end{equation}

In the following numerical tests, we use the above spectral method for space discretization. 

\subsection{2D Allen--Cahn with polynomial potential}
Consider the AC equation \eqref{1.1} with polynomial potential, where $\eps = 0.1$ and the $2\pi$-periodic domain $\Omega = [0,2\pi]^2$. We take the initial data $u_0$ as
\begin{equation}
u_0(x,y) = 0.05\sin(x)\sin(y).
\end{equation}
We use $N\times N =  512\times 512$ Fourier modes for the space discretization. 

Since the exact PDE solution is not available, we take a small splitting step  $\tau = 10^{-4}$, to obtain an ``almost exact'' solution $u_{\mathrm {ex}}$ at time $T = 20$.
Then, we take several different splitting steps $\tau = \frac{1}{10}\times2^{-k}$ with $k = 0,1,\ldots,4$ and obtain corresponding numerical solutions at $T=20$. 
The $\ell_2$-errors between these solutions and the ``almost exact'' solution are summarized in Table \ref{tab2}. Reassuringly,  it is observed that the convergence rate is about $2$, i.e. the scheme
has second order in-time accuracy.

\begin{table}[htb!]
\renewcommand\arraystretch{1.5}
\begin{center}
\def\temptablewidth{0.95\textwidth}
\caption{$\ell_2$-errors of numerical solutions to the AC equation \eqref{1.1} with polynomial potential at time $T = 20$ for different splitting steps.}\label{tab2}
\vspace{-0.2in}
{\rule{\temptablewidth}{1pt}}
\begin{tabular*}{\temptablewidth}{@{\extracolsep{\fill}}cccccc}
   $\tau$   &$\frac1{10}$     &$\frac1{20}$     &$\frac1{40}$
   &$\frac1{80}$  & $\frac1{160}$ \\  \hline
  $\ell_2$-error   & $9.367\times 10^{-4}$   & $2.345\times 10^{-4}$   & $5.865\times 10^{-5}$ & $1.466\times 10^{-5}$ & $3.665\times 10^{-6}$\\[3pt]
rate  & -- &$1.998$ &$1.999$ &$2.000$ & $2.000$
\end{tabular*}
{\rule{\temptablewidth}{1pt}}
\end{center}
\end{table}
   
In Figure \ref{fig:E_poly}  we plot the standard energy versus the modified energy as a function of time. The time step is  $\tau = 0.01$.
It is observed that the standard energy and the modified energy coincide approximately, and
they both decay monotonically in time.

\begin{figure}[!]
\includegraphics[trim = {0in 0.in 0.5in 0.2in},clip,width = 0.5\textwidth]{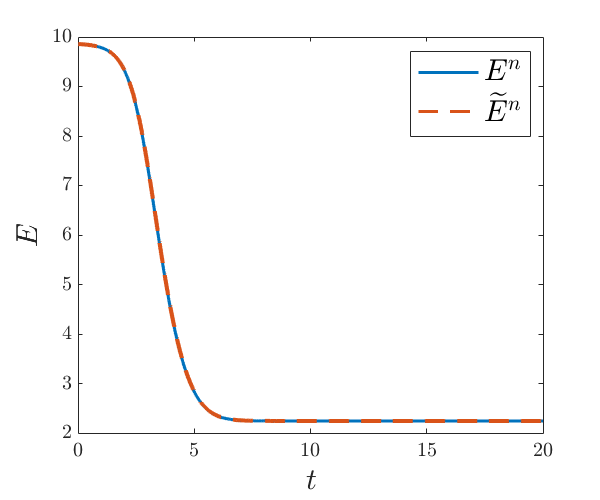}
\caption{Standard energy $E^n$ and modified energy $\widetilde E^n$ w.r.t. time for the Strang splitting method, with splitting step $\tau = 0.01$ and number of Fourier modes $512\times 512$.}\label{fig:E_poly}
\end{figure}

\subsection{2D AC with the logarithmic potential}
Consider the AC equation \eqref{Eu_3.1} with logarithmic potential, where $\eps = 0.01$, $\theta_c = 1$ and $\theta = \frac14$. 
The spatial domain is the two-dimensional $2\pi$-periodic torus  $\Omega = [0,2\pi]^2$.
We take the initial condition $u_0$ as
\begin{equation}
u_0(x,y) = 0.5 \left[\chi \left((x-\pi)^2+(y-\pi)^2\le 1.2\right) - 0.5\right],
\end{equation}
where $\chi$ is the characteristic function.
We employ the RK-based Strang splitting method \eqref{Eu_3.12} to solve this equation. The tolerance threshold of the Newton iterative solver is set to be $10^{-12}$.
We use the standard Fourier spectral method with $512\times 512$ Fourier modes for the space discretization.

As a first step, we test the convergence rate of the  RK-based Strang splitting method.  
In Table \ref{tab2}, we show the $\ell_2$-errors of the numerical solution at $T = 1$, where the parameter $a$ in the PR-RK method is set to $a = 1+\frac{\sqrt 2}{2}$.
As before, the ``exact'' solution is taken as the numerical solution when $\tau = 10^{-4}$. 
It can be observed that the convergence order is about $2$.
\begin{table}[htb!]
\renewcommand\arraystretch{1.5}
\begin{center}
\def\temptablewidth{1\textwidth}
\caption{$\ell_2$-errors of numerical solutions at time $T = 1$ to the AC equation with logarithmic potential \eqref{Eu_3.1} for different splitting steps, computed by the RK-based Strang splitting method with $a = 1+\frac{\sqrt 2}{2}$.}\label{tab3}
\vspace{-0.2in}
{\rule{\temptablewidth}{1.1pt}}
\begin{tabular*}{\temptablewidth}{@{\extracolsep{\fill}}cccccc}
   $\tau$   &$\frac1{10}$     &$\frac1{20}$     &$\frac1{40}$
   &$\frac1{80}$  & $\frac1{160}$ \\  \hline
  $\ell_2$-error   & $2.245\times 10^{-2}$   & $4.935\times 10^{-3}$   & $1.160\times 10^{-3}$ & $2.815\times 10^{-4}$ & $6.933\times 10^{-5}$ \\[3pt]
rate  & -- &$2.186$ &$2.088$ &$2.043$ & $2.022$ 
\end{tabular*}
{\rule{\temptablewidth}{1pt}}
\end{center}
\end{table}

Secondly, we test the convergence rate for an interesting case of $a = \frac12+\frac{\sqrt{3}}{6}$, where the PR-RK method in Table \ref{tab:PR} becomes the Crouzeix's third order RK method. 
In this case the approximation error of nonlinear solution operator becomes $\mathcal O(\tau^4)$, i.e.,
\begin{equation}
\mathcal S_{\mathcal N}(\tau) \tilde u^n = \widetilde{\mathcal S}_{\mathcal N}(\tau)  \tilde u^n  +\mathcal O(\tau^4).
\end{equation}
On the other hand, the overall error of the method \eqref{eq:PRRK} is still second order in time.
Interestingly, the numerical results in Table \ref{tab4} show that the convergence rate for $a = \frac12+\frac{\sqrt{3}}{6}$ appears to be higher than the corresponding case of $a = 1+\frac{\sqrt 2}{2}$ in Table \ref{tab3}.  This is probably due to the inaccuracy of the reference solution
which was taken as the $\tau=10^{-4}$-almost exact-solution.

\begin{table}[htb!]
\renewcommand\arraystretch{1.5}
\begin{center}
\def\temptablewidth{1\textwidth}
\caption{$\ell_2$-errors of numerical solutions at time $T = 1$ to the AC equation with logarithmic potential \eqref{Eu_3.1} for different splitting steps, computed by the RK-based Strang splitting method with $a = \frac12+\frac{\sqrt{3}}{6}$.}\label{tab4}
\vspace{-0.2in}
{\rule{\temptablewidth}{1.1pt}}
\begin{tabular*}{\temptablewidth}{@{\extracolsep{\fill}}ccccccc}
   $\tau$   &$\frac1{10}$     &$\frac1{20}$     &$\frac1{40}$
   &$\frac1{80}$  & $\frac1{160}$ & $\frac{1}{320}$\\  \hline
  $\ell_2$-error   & $9.440\times 10^{-5}$   & $1.132\times 10^{-5}$   & $1.392\times 10^{-6}$ & $1.750\times 10^{-7}$ & $2.286\times 10^{-8}$ & $3.302\times 10^{-9}$ \\[3pt]
rate  & -- &$3.060$ &$3.023$ &$2.992$ & $2.936$ & $2.792$
\end{tabular*}
{\rule{\temptablewidth}{1pt}}
\end{center}
\end{table}

Finally, we test  the maximum principle and the energy dissipation of the RK-based Strang splitting method.
We set $a = 1+\frac{\sqrt 2}{2}$ and $\tau = 0.01$, so that the restrictions in Theorem \ref{thm:max} and \ref{thm:ene2} are satisfied. 
Numerical solutions up to $t = 10$ are illustrated in Figure \ref{fig:phase_log}.
It can be observed that $\|u\|_\infty$ is always less than $u_*\approx 0.99933$, i.e., the maximum principle holds.
In Figure \ref{fig:E_log}, we plot the standard energy $E^n$ w.r.t. time which clearly decays
in time.  Note that the modified energy $\overline E^n$ is implicit in this case and is not plotted here. 

\begin{figure}[!]
\includegraphics[trim = {0in 0.8in 0.in 0},clip,width = 0.99\textwidth]{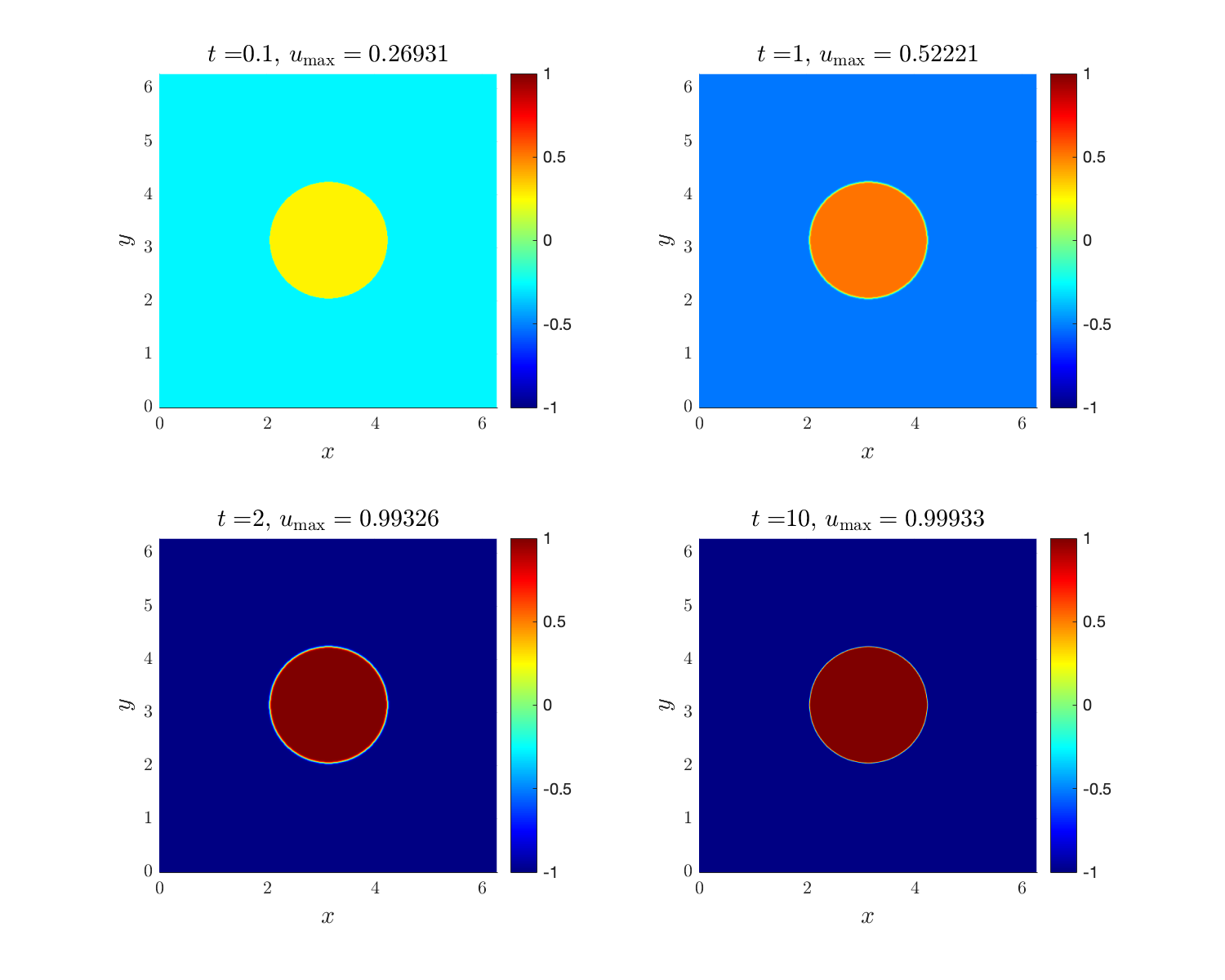}
\caption{Numerical solution to the AC equation with logarithmic potential computed by the RK-based Strang splitting method with splitting step $\tau = 0.01$ and number of Fourier modes $512\times 512$. $u_{\mathrm{max}}$ denotes the maximal absolute value of $u$.}\label{fig:phase_log}
\includegraphics[trim = {0in 0.in 0.5in 0.2in},clip,width = 0.5\textwidth]{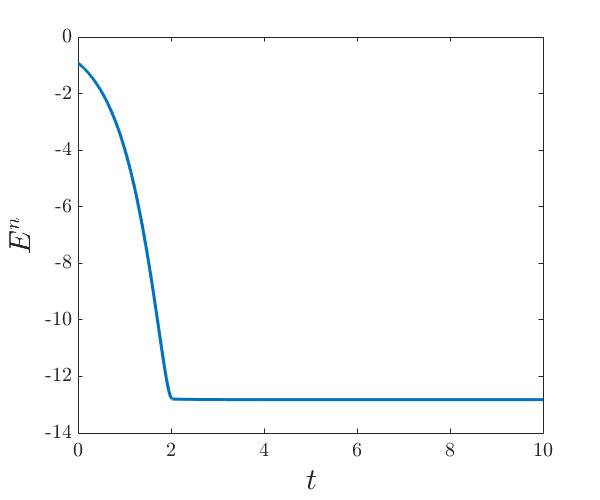}
\caption{Standard energy w.r.t. time for the RK-based Strang splitting method \eqref{Eu_3.12} with splitting step $\tau = 0.01$ and number of Fourier modes $512\times 512$.}\label{fig:E_log}
\end{figure}

\subsection{Seven circles}
{\em Consider the AC equation \eqref{Eu_3.1} with  $\eps = 0.1$, $\theta_c = 1$ and $\theta = \frac14$. 
The domain is the two-dimensional $2\pi$-periodic torus $\Omega = [0,2\pi]^2$.
The initial condition consists of seven circles with centers and radii given in Table \ref{tab:xyr}:
\begin{equation}\label{Eu_4.11}
u_0(x,y) = -1 + \sum_{i=1}^7 f_0\left( \sqrt{(x-x_i)^2+(y-y_i)^2} - r_i\right),
\end{equation}
where
\begin{equation}
f_0(s) = \left\{
\begin{aligned}
& 2 e^{-\varepsilon^2/s^2} &&\mbox{if } s<0,\\
& 0 && \mbox{otherwise.}
\end{aligned}
\right.
\end{equation}
}

\begin{table}[htb!]
\renewcommand\arraystretch{1.3}
\begin{center}
\def\temptablewidth{0.8\textwidth}
\caption{Centers $(x_i,y_i)$ and radii $r_i$ in the initial condition \eqref{Eu_4.11}.}\label{tab:xyr}
{\rule{\temptablewidth}{1.1pt}}
\begin{tabular*}{\temptablewidth}{@{\extracolsep{\fill}}c|ccccccc}
   $i$   &1 & 2     &3   &4  & 5 & 6 & 7 \\  \hline
  $x_i$  & $\pi/2$   & $\pi/4$   & $\pi/2$ & $\pi$ & $3\pi/2$ & $\pi$ & $3\pi/2$ \\[3pt]
$y_i$  & $\pi/2$   & $3\pi/4$   & $5\pi/4$ & $\pi/4$ & $\pi/4$ & $\pi$ & $3\pi/2$ \\[3pt]
$r_i$  & $\pi/5$   & $2\pi/15$   & $2\pi/15$ & $\pi/10$ & $\pi/10$ & $\pi/4$ & $\pi/4$ \\[3pt]
\end{tabular*}
{\rule{\temptablewidth}{1.1pt}}
\end{center}
\end{table}

We  use the RK-based Strang splitting method with $a=1+\frac{\sqrt 2}2$ and $\tau=0.01$ to solve this equation with the Newton iterative solver.
 To achieve mediocre accuracy  the tolerance threshold for the Newton iteration is set as $10^{-12}$ which is close to the machine precision.
We employ the spectral method with $512\times 512$ Fourier modes for the space discretization. 
The evolution of phase field is illustrated in Figure \ref{fig:sol_7circ}, where the annihilation of the circles  take place gradually in time.
The corresponding energy evolution is recorded in Figure \ref{fig:energy_7circ}.

\begin{figure}[!]
\includegraphics[trim = {0in 0.8in 0.in 0},clip,width = 0.99\textwidth]{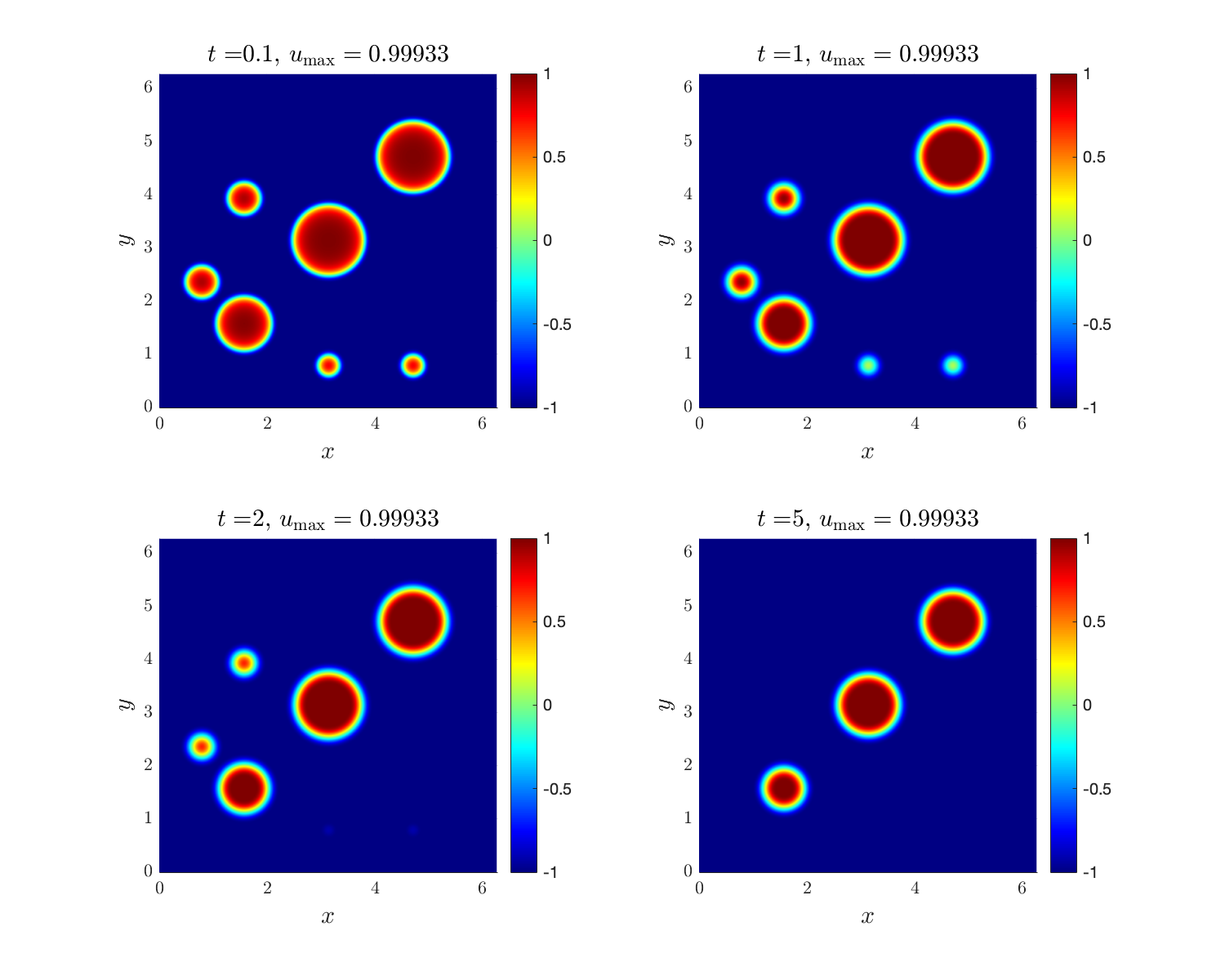}
\caption{Numerical solution of the seven circles example computed by the RK-based Strang splitting method with splitting step $\tau = 0.01$ and number of Fourier modes $512\times 512$. $u_{\mathrm{max}}$ denotes the maximal absolute value of $u$.}\label{fig:sol_7circ}
\includegraphics[trim = {0in 0.in 0.5in 0.2in},clip,width = 0.5\textwidth]{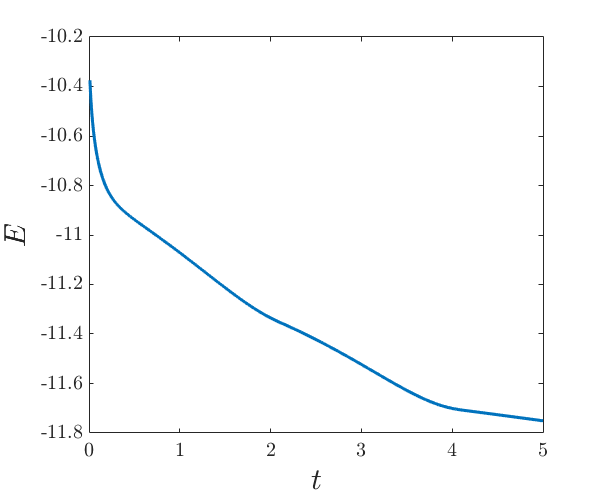}
\caption{Standard energy w.r.t. time for the RK-based Strang splitting method \eqref{Eu_3.12} in the seven circles example  with splitting step $\tau = 0.01$ and number of Fourier modes $512\times 512$.}\label{fig:energy_7circ}
\end{figure}

\section{Conclusion}
In this work we investigated a class of second-order Strang splitting methods for Allen-Cahn equations with polynomial and logarithmic nonlinearities. For the polynomial case we compute both the linear and the nonlinear propagators explicitly.  Unconditional stability is established for any time step $\tau>0$. For a judiciously modified energy which coincides with the classical energy up to $O(\tau)$,
we show strict energy dissipation and obtain uniform control of higher Sobolev norms. 
For the logarithmic potential case, since the continuous-time nonlinear propagator no longer enjoys
explicit analytic treatments, we adopted a second order in time two-stage implicit Runge--Kutta (RK) nonlinear propagator together with an efficient Newton iterative solver.   We establish a sharp maximum principle which ensures phase separation. 
We prove a new modified energy dissipation law under very mild restrictions on the time step. 
The methods introduced in this work can be generalized to many other models including nonlocal Allen-Cahn models, Cahn--Hilliard equations, general  drift-diffusion systems and nonlinear parabolic
systems.

\section*{Acknowledgements}
 The research of C. Quan is supported by NSFC Grant 11901281, the Guangdong Basic and Applied Basic Research Foundation (2020A1515010336), and the Stable Support Plan Program of Shenzhen Natural Science Fund (Program Contract No. 20200925160747003).


\end{document}